\documentclass[11pt, reqno]{amsart}
\usepackage{hyperref}
\usepackage{cmap}                        
\usepackage[cp1251]{inputenc}            
\usepackage[english]{babel}
\usepackage[left=1in,right=1in,top=1in,bottom=1.5in]{geometry} 
\usepackage{amssymb,amsmath, amsthm, amscd,ifthen}
\usepackage{graphicx}
\usepackage{epigraph, xcolor, hyperref}

\newtheorem*{thm*}{Theorem}
\newcommand{\ff}{{\mathcal F}}

\newcommand{\G}{{\mathcal G}}

\newtheorem*{cla*}{Claim}

\newtheorem{thm}{Theorem}
\newtheorem{gypo}{Conjecture}

\newtheorem{lem}[thm]{Lemma}
\newtheorem{cla}[thm]{Claim}

\date{}

\newtheorem{defn}[thm]{Definition}

\date{}
\title{Intersection problems and a correlation inequality for integer sequences}
\author{Peter Frankl}\address{R\'enyi Institute, Budapest, Hungary; Email: {\tt peter.frankl@gmail.com}}

\author{Andrey Kupavskii}
\address{Moscow Institute of Physics and Technology, Russia; Email: {\tt kupavskii@ya.ru}.} 

\begin{document}
\maketitle
\begin{abstract} Let us consider a collection $\G$ of codewords of length $n$ over an alphabet of size $s$. Let $t_1,\ldots, t_s$ be nonnegative integers.  What is the maximum of $|\G|$ subject to the condition that any  two codewords should have at least $t_i$ positions where both have letter $i$ ($1\le i\le s$). In the case $s=2$ it is a longstanding open question. Quite surprisingly we obtain an almost complete answer for $s\ge 3$. The main tool is a correlation inequality.
\end{abstract}

For an integer $s\ge 2$ let $[s]=\{1,2,\ldots, s\}$ be the standard $s$-set. Let $[s]^n$ denote the collection of all $s^n$ integer sequences $\vec a = (a_1,\ldots, a_n)$ with $a_i\in [s]$. Calling $[s]$ the alphabet and each $\vec a$ a codeword, $[s]^n$ is the central object in {\it coding theory} (cf., e.g., \cite{L}).

Another way to look at these sequences is to consider $n$ pairwise disjoint sets $Y_1,\ldots, Y_n$ and consider the complete $n$-partite $n$-graph $Y_1\times\ldots \times Y_n$ consisting of all vectors $\vec y = (y_1,\ldots, y_n)$, $y_i\in Y_i$. If $|Y_1| = \ldots = |Y_n|=s$, it is essentially the same combinatorial object.

For two vectors $\vec y$ and $\vec z$ let us define their {\it intersection} ({\it meet}) $\vec y\wedge \vec z = (w_1,\ldots, w_n)$, where
$$w_i = \begin{cases}y_i, \text{ if } y_i = z_i\\
0, \text{ if } y_i\ne z_i.\end{cases}$$

A family $\ff\subset Y_1\times \ldots \times Y_n$ is called {\it $t$-intersecting} if for all $\vec y, \vec z\in \ff$,  $\vec y\wedge \vec z$ has at least $t$ non-zero coordinates. Note that in the $n$-partite setting it is the same as requiring $|\{y_1,\ldots, y_n\}\cap \{z_1,\ldots, z_n\}|\ge t$.

In the case of $t=1$ we omit the $1$ and say simply {\it intersecting}. Let us recall two, by now classical results, analogous to the famous Erd\H os--Ko--Rado Theorem \cite{EKR}.

\begin{thm}[Deza and Frankl \cite{DF}]\label{thmold1} Suppose that $\ff\subset Y_1\times \ldots \times Y_n$ is intersecting and $|Y_1|\le |Y_i|$, $i = 2,\ldots, n$. Then
\begin{equation}\label{eq1} |\ff|\le \prod_{2\le i\le n} |Y_i|.\end{equation}
\end{thm}
Let us note that the case $|Y_1| = \ldots = |Y_n|$ goes back to Meyer \cite{M}.
\begin{thm}[Frankl and F\"uredi \cite{FF6}]\label{thmold2}
  Suppose that $s>t\ge 1$ and $\ff\subset [s]^n$ is $t$-intersecting. then
  \begin{equation}\label{eq2}
    |\ff|\le s^{n-t}.
  \end{equation}
\end{thm}

Note that both \eqref{eq1} and \eqref{eq2} are easily seen to be best possible.
Even that the case $s=2,t=1$ of Theorem~\ref{thmold2} is trivial it is closely related to a classical result in extremal set theory.
\begin{thm}\label{thmold3} Suppose that $\G\subset 2^X$ satisfies $G\cap G'\ne \emptyset$ and $G\cup G'\ne X$ for all $G,G'\in \G$. Then
\begin{equation}\label{eq3}
  |\G|\le 2^{|X|-2}.
\end{equation}
\end{thm}
This important result  was proved nearly simultaneously by several sets of authors (\cite{DL}, \cite{S},\ldots ). The authors of the present note recently gave an alternative proof of Theorem~\ref{thmold3} along with some generalizations.
Theorem \ref{thmold3} can be restated in the terminology of integer sequences as follows.

Suppose that $\ff\subset 2^{[n]}$ and for all $\vec y,\vec z\in \ff$ there exist $i,j\in [n]$ so that $y_i=z_i=1$ and $y_j=z_j=2$. Then $|\ff|\le 2^{n-2}$. Let us generalize this situation by introducing the notion of {\it $(t_1,t_2,\ldots, t_s)$-intersecting} property.
\begin{defn}
  Let $t_1,\ldots, t_s$ be non-negative integers and $\ff\subset [s]^n$. then $\ff$ is called $(t_1,\ldots, t_s)$-intersecting if for all $\vec y,\vec z\in \ff$, $\vec y\wedge \vec z$ has at least $t_\ell$ coordinates equal to $\ell$ for all $1\le \ell \le s$.
\end{defn}

Note that $\ff\ne \emptyset$ implies $t_1+\ldots t_s\le n$ and also that any $(t_1,\ldots, t_s)$-intersecting family is $(t_1+\ldots t_s)$-intersecting.

The simplest version of the main result of the present paper is the following

\begin{thm}\label{thmmain}
  Let $s\ge 2$ and suppose that $\ff\subset [s]^n$ is $(t_1,\ldots,t_s)$-intersecting, $n\ge t_1+\ldots+t_s$, $0\le t_i<s$ for each $1\le i\le s$. Then
  \begin{equation}\label{eq4}
  |\ff|\le s^{n-\sum_{i\in[s]}t_i}.
  \end{equation}
\end{thm}

\section{A correlation inequality for integer sequences}

Let $\ff\subset [s]^n$ and $P$ a proper subset of $[s]$. Let us introduce the relationship $<_P$ by setting $(x_1,\ldots, x_n)<_P (y_1,\ldots, y_n)$ iff for each $1\le i\le n$ either $x_i=y_i$ or $x_i\notin P.$ Note that for $y,z\notin P$, both $(y)<_P (x)$ and $(x)<_P (y)$ hold! In human language, to make a `larger' vector in $<_P$ we can change any coordinates with entries not in $P$ to arbitrary entries.
\begin{defn}
  The family $\ff\subset [s]^n$ is called {\it $P$-complete} if $\vec x<_P\vec y$ and $\vec x\in \ff$ always imply $\vec y\in \ff$.
\end{defn}

\begin{thm}\label{thm2}
  Supposet that $P,Q \subset [s]$ are non-empty and disjoint. Consider two families $\ff,\G\subset [s]^n$ such that $\ff$ is  $P$-complete and $\G$ is a $Q$-complete. Then
  \begin{equation}\label{eq5}
    |\ff||\G|\ge s^n|\ff\cap \G|.
  \end{equation}
\end{thm}
Note that \eqref{eq5} is equivalent to
$$\frac{|\ff\cap\G|}{s^n}\le \frac{|\ff|}{s^n}\cdot \frac{|\G|}{s^n}. $$
That is, \eqref{eq5} shows that the probabilistic events $\vec x\in \ff, \vec x\in \G$ are negatively correlated w.r.t. the uniform measure on $[s]^n$.

The case $[s]=\{1,2\}$, $P=\{1\}$, $Q=\{2\}$ of \eqref{eq5} is equivalent to the Kleitman--Harris correlation inequality (cf. \cite{Kl}, \cite{H}) for up-sets and down-sets.

\begin{proof}
  The proof is by induction on $n$. Consider the  base case, $n=1$, first. Then $\ff,\G$ are subsets of $[s]$. If $\ff\cap \G=\emptyset$ then \eqref{eq5} holds. Assume that $x\in \ff\cap \G$. If $x\notin P\cup Q$ then $\ff=\G=[s]$ and therefore \eqref{eq5} follows. Say, $\ff\neq [s]$. Then $\ff\subset P$. Using $x\in \ff\cap \G$ and $P\cap Q=\emptyset$, $\G = [s]$ follows. Thus $|\ff| = |\ff\cap \G| = |\ff||\G|/s$, proving \eqref{eq5}.

  Next, assume that \eqref{eq5}  holds for $n$ and let us prove it for $n+1$. Take two families $\ff,\G\in[s]^{n+1}$  and consider the families $\ff_i$, $\G_i$, $i\in [s]$, where
  $$\ff_i = \{(x_1,\ldots, x_n)\in [s]^n: (x_1,\ldots, x_n,i)\in \ff\}$$
  and $\G_i$ is defined analogously. Set also $f_i = |\ff_i|$ (and similarly $g_i = |\G_i|$). Note the relations $f_i\le f_j$ if $i\ne P$ (implying $f_i=f_j$ if $i,j\notin P$) and $g_i\le g_j$ if $i\notin Q$.

  Let $f$ ($g$) be the common value of $f_i$ ($g_i$) for $i\notin P$ ($i\notin Q$), respectively. Note that
  $$|\ff|=f_1+\ldots +f_s= (s-|P|)f+\sum_{i\in P} (f_i-f),$$
  $$|\G|=g_1+\ldots +g_s= (s-|Q|)g+\sum_{i\in Q} (g_i-g),$$
  and also $(f_i-f)(g_i-g)=0$ by $P\cap Q=\emptyset$ for all $i$.
  By the induction hypothesis, we have
  $$s^{n+1}|\ff\cap \G|= s^{n+1}\sum_{i\in [s]}|\ff_i\cap \G_i|\le s \sum_{i\in [s]}f_ig_i=s^2fg +sf\sum_{i\in Q}(g_i-g)+sg\sum_{i\in P}(f_i-f).$$
On the other hand,
  \begin{align*}|\ff||\G| =& (\sum_{i\in [s]}f_i)(\sum_{i\in [s]}g_i)=(sf+\sum_{i\in P}(f_i-f))(sg+\sum_{i\in Q}(g_i-g))\\ =& s^2fg +sf\sum_{i\in Q}(g_i-g)+sg\sum_{i\in P}(f_i-f)+(\sum_{i\in P}(f_i-f))(\sum_{i\in Q}(g_i-g)).\end{align*}
  Since the last term is non-negative, the proof of \eqref{eq5} is complete.
  \end{proof}

\section{The proof of Theorem~\ref{thmmain}}
For $\vec t = (t_1,\ldots, t_s)$ let us define
$$p(n,s,\vec t) = \max\big\{|\ff|/s^n: \ff\subset [s]^n \text{ is } \vec t\text{-intersecting}\big\}.$$
\begin{lem}
  For $1\le r<s$,
  \begin{equation}\label{eq6}
    p(n,s,\vec t)\le p(n,s,(t_1,\ldots, t_r,0,0,\ldots, 0)p(n,s,0,\ldots, 0,t_{r+1},\ldots,t_s).
  \end{equation}
\end{lem}
\begin{proof}
  Suppose that $\ff\subset [s]^n$ is $\vec t$-intersecting, $|\ff|$ is maximal. For $P\subset [s]$ let us define the $P$-complete family $\ff(P)$ by $$\ff(P) = \{\vec z\in [s]^n: \exists \vec y\in \ff, \vec y<_P \vec z.\}$$
  In human words $\ff(P)$ is obtained from $\ff$ by adding all possible sequences obtained from some sequence in $\ff$ by changing entries $y\notin P$ to arbitrary $y\in [s]$. The $P$-completeness of $\ff(P)$ should be clear.

  Since no entry $Y_\ell$ of $(y_1,\ldots, y_n)\in \ff$ satisfying $y_\ell\in P$ is ever changed for each $i\in P$, $\ff(P)$ is $(0,0,\ldots, t_i,\ldots,0)$-intersecting.

Consequently,
\begin{align*}|\ff([r])|/s^n&\le p\big(n,s,(t_1,\ldots, t_r,0,\ldots, 0)\big) \ \ \ \text{and}\\
|\ff([r+1,s])|/s^n&\le p\big(n,s,(0,\ldots,0,t_{r+1},\ldots, t_s)\big).
\end{align*}
  Since $\ff\subset \ff([r])\cap \ff([r+1,s])$ is obvious, \eqref{eq6} follows from \eqref{eq5}.
\end{proof}

Now the proof of Theorem~\ref{thmmain} is easy. In view of Theorem~\ref{thmold3}, $p(n,s,(0,\ldots,0,t_i,\ldots,0))=s^{-t_i}$ for each $1\le i\le s$ and $0\le t_i<s$. Applying repeatedly \eqref{eq6}, we obtain
$$p(n,s,(t_1,\ldots,t_s))\le s^{-t_1}p(n,s,(0,t_2,\ldots, t_s))\le s^{-t_1-t_2}p(n,s,(0,0,t_3,\ldots,t_s))\le \ldots\le s^{-t_1-\ldots -t_s}.$$

\section{The case $s=2$}
The results for $s=2$ that we mentioned in the introduction might appear meager. They concern the case $t=1$ only. There is a good reason for it.

\begin{cla}
  We have $\lim_{n\to \infty} p(n,2,(t_1,t_2))=\frac 14$ for any fixed positive integer vector $t_1,t_2$.
\end{cla}
\begin{proof}
  Since $p(n,2,(t_1,t_2)\le 2^{-2}$ follows from \eqref{eq3}, we only need to show the opposite inequality by construction.

  Let $X_1,X_2\subset[n]$ with $X_1\cap X_2\ne \emptyset$ and $|X_1|=n_1, |X_2|=n_2$. Define
  $$\mathcal K(X_1,X_2,\vec t)=\Big\{(y_1,\ldots, y_n)\in 2^{[n]}: \big|\{j: j\in X_i \text{ and }y_j=i\}\big|\ge \frac {n_i+t_i}2\text{ for } i=1,2\Big\}.$$
  For $t_i$ fixed and $n_i\to \infty$, $|\mathcal K(X_1,X_2,\vec t)|=(\frac 12+o(1))^22^n$ should be clear as well as the $(t_1,t_2)$-intersecting property.
\end{proof}

In the Ph.D. dissertation of the first author \cite{Fphd} the following conjecture was stated.
\begin{gypo} Suppose that $\vec t = (t_1,t_2)$ is a positive integer vector with $t_1+t_2\le n$. Assume that $\ff\subset 2^{[n]}$ is $\vec t$-intersecting. Then
\begin{equation}\label{eq7}
  |\ff|\le \max_{X_1,X_2}|\mathcal K(X_1,X_2,\vec t)|.
\end{equation}
\end{gypo}
Let us mention that the maximum of the RHS is achieved for $X_1,X_2$ satisfying $|X_i|\equiv t_i\mod 2$ and $|X_1|+|X_2| = n$ or $n-1$. The case $\min\{t_1,t_2\}=1$ was proved in \cite{F75}.

Set $q=n-t_1-t_2$. For the cases $q=0$ or $1$, $\max|\ff|=2^q$ is easy to verify. Bang, Sharp and Winkler \cite{BSW} proved \eqref{eq7} for $q=2$ or $3$. Recently, Wang and the first author \cite{FW} settled the case $n>n_0(q)$ by proving \eqref{eq7} for $n>(q+1)^3$ and also for $q=4$.

Let us also mention that the cases $(t,0)$ and $(0,t)$ are equivalent to the classical Katona Theorem \cite{Ka}.

\section{The general case}
Suppose that $s\ge 3$. There is a striking difference to the case $s=2$. Let us explain it by considering the $(t,0,\ldots,0)$ case.

Suppose that $X\subset [n]$, $|X|\equiv t \mod 2$. The construction
$$\mathcal L(X) =\Big\{(y_1,\ldots, y_n)\in [s]^n: |\{i\in X: y_i=1\}|\ge \frac{|X|+t}2\Big\}$$
is clearly $(t,0,\ldots, 0)$-intersecting. However, the overwhelming majority of sequences $\vec y\in [s]^n$ has around $n/s$ coordinates equal to $1$. Hence for $|X|>(\frac 1s+\epsilon)n$ with constant $\epsilon>0$, $|\mathcal L(X)|/s^n$ tends to $0$ as $n\to \infty$. In fact, this ratio is at most $e^{-2\epsilon^2 n/s^2}$ using standard Chernoff-type concentration inequalities (see, e.g., \cite[Theorem A.1.4]{AS}).

Let us define the $p$-biased measure $\mu_p$ on $2^{[n]}$ by setting $\mu_p(F) = p^{|F|}(1-p)^{n-|F|}$ for all $F\in [n]$. Filmus \cite{Fil} proved a Complete $t$-intersection Theorem for families in $2^{[n]}$ with a $p$-biased measure. In order to state it, we need some definitions. Put $r^*$ to be the largest integer satisfying $r^*\le (n-t)/2$. For each $r\le r^*$ define
$$\ff_{t,r} = \{A\subset [n]: |A\cap [t+2r]|\ge t+r\}.$$
Also, define
\begin{equation}\label{eq8}
  w(n,t,p) = \begin{cases}
               \mu_p(\ff_{t,r}), & \mbox{if } \frac r{t+2r-1}\le p\le  \frac{r+1}{t+2r+1} \text{ for some }r<r^*,\\
               \mu_p(\ff(t,r^*)), & \mbox{if } \frac{r^*}{t+2r^*-1}\le p.
             \end{cases}
\end{equation}
Then the theorem of Filmus states
\begin{thm}\label{thmold4}
If $\ff\subset 2^{[n]}$ is $t$-intersecting with $t\ge 2$ then $\mu_p(\ff)\le w(n,p,t)$. 
\end{thm}

We are particularly interested in the case $p=1/s$. Doing simple algebraic manipulations, we see that for $p=1/s$ the right inequality on $p$ in the first line of \eqref{eq8} is equivalent to
\begin{equation}\label{eq9}t+2r\ge \frac s{s-2}t-\frac {2s-2}{s-2}=t+2\cdot \frac{t-s+1}{s-2}.\end{equation}
That is, if $t+2\big\lceil\frac{t-s+1}{s-2}\big\rceil\le m$ for some integers $m$ and $n\ge m$, then $w(n,t,1/s)=\mu_{1/s}(\ff_{t,r})$ for some $r$ such that $t+2r\le m$. Moreover, note that, for any integer $n', n'\ge m$, we have $w(n',t,1/s) = w(n,t,1/s)$ in these conditions.

The case of $p=\frac 1 s$ is actually very relevant to our question for the following reason. Given a $(t,0,\ldots,0)$-intersecting family $\G\subset [s]^n$, we can form a family $\mathcal P_1(\G)\subset 2^{[n]}$, where $\mathcal P_i(\G)$ is defined as follows:
$$\mathcal P_i(G) = \big\{G\subset [n]: G = \{i: y_i=i\}\text{ for some } (y_1,\ldots, y_n)\in \G\big\}.$$
If we assume that $\G$ is $\{1\}$-complete (which holds obviously for the largest $(t,0,\ldots, 0)$-intersecting family in $[s]^n$), then it is straightforward to see that $$|\G|/s^n=\mu_{1/s}(\mathcal P_1(G)).$$
The following theorem is the main result of this note.
\begin{thm}\label{thmmain2}
  Suppose that $n\ge s\ge 3$ and $t_1,\ldots, t_s$ are non-negative integers. Assume that $\G\subset [s]^n$ is $(t_1,\ldots, t_s)$-intersecting of maximum size and $\sum_{i=1}^s \big(t_i +2\big\lceil\frac{t_i-s+1}{s-2}\big\rceil\big)\le n$ (which is implied by $\sum_{i=1}^s t_i\le \frac {s-2}sn+s$). Then
  $$\frac{|\G|}{s^n} = \prod_{i=1}^s w(n,t_i,1/s).$$
\end{thm}

\begin{proof}
  The upper bound is akin to the proof of Theorem~\ref{thmmain}. Using \eqref{eq6}, we get that
  $$\frac{|\G|}{s^n}\le \prod_{i=1}^sp(n,s,(0,\ldots,0,t_i,0\ldots,0)=\prod_{i=1}^s w(n,t_i,1/s).$$
We go on to the proof of the lower bound. Define $r_i$ so that $w(n,t_i,1/s) = \mu_{1/s}(\ff_{t,r_i})$ and take a partition of $[n]$ into $s$ pairwise disjoint sets $X_1,\ldots, X_s$ so that $|X_i|=t_i+2r_i$ for $i=1,\ldots, s-1$ and $|X_s| = n-\sum_{i=1}^{s-1}(t_i+2r_i)\ge t_s+2r_s$. Here, the last inequality holds by  \eqref{eq9} and the subsequent analysis. For each $i\in [s]$, let $\ff^{i}\subset [s]^{X_i}$ be an $\{i\}$-complete family of words so that $\mathcal P_i(\ff^i)\subset 2^{X_i}$ is isomorphic to $\ff_{t,r_i}$. Note that $\ff_{t,r_i}$ `fits' into $2^{X_i}$ by the choice of $|X_i|$.   For a word $\vec y$ define $\vec y|_{X_i}$ as the part of the vector $\vec y$ induced on coordinates from $X_i$. Let the family $\ff\subset [s]^n$ be defined as follows:
$$\ff = \big\{\vec y=(y_1,\ldots,y_n)\in [s]^n: \text{ for all }i\in [s] \text{ and }i\in X_i\text{ we have } \vec y|_{X_i}\in \ff^{i}\big\}.$$
It is easy to see that, first, $\ff$ is $(t_1,\ldots,t_s)$-intersecting and, second, $$|\ff| = \prod_{i\in [s]}|\ff^i| = s^n \prod_{i=1}^s w(|X_i|,t,1/s) = s^n \prod_{i=1}^s w(n,t,1/s).$$
This completes the proof of the theorem.
\end{proof}

\end{document}